\documentclass[12pt, reqno]{amsart}
\usepackage{ amsmath,amsthm, amscd, amsfonts, amssymb, graphicx, color}
\usepackage[bookmarksnumbered, colorlinks, plainpages]{hyperref}
\newtheorem{thm}{Theorem}[section]
\newtheorem{cor}[thm]{Corollary}
\newtheorem{lem}[thm]{Lemma}

\numberwithin{equation}{section}

\begin{document}

\title{New characterizations of g-Drazin inverse in a Banach algebra}

\author{Huanyin Chen}
\author{Marjan Sheibani$^*$}
\address{
Department of Mathematics\\ Hangzhou Normal University\\ Hang -zhou, China}
\email{<huanyinchen@aliyun.com>}
\address{Women's University of Semnan (Farzanegan), Semnan, Iran}
\email{<sheibani@fgusem.ac.ir>}

 \thanks{$^*$Corresponding author}


\begin{abstract}
In this paper, we present a new characterization of g-Drazin inverse in a Banach algebra. We prove that an element $a$ is a Banach algebra has g-Drazin inverse if and only if there exists $x\in \mathcal{A}$ such that $ax=xa, a-a^2x\in \mathcal{A}^{qnil}.$ As application, we obtain
the sufficient and necessary conditions for the existence of the g-Drain inverse for certain $2\times 2$ anti-triangular matrices over a Banach algebra.
These extend the results of Koliha (Glasgow Math. J., {\bf 38}(1996), 367--381), Nicholson (Comm. Algebra, {\bf 27}(1999), 3583--3592 and Zou et al. (Studia Scient. Math. Hungar., {\bf 54}(2017), 489--508).
\end{abstract}

\maketitle

\section{Introduction}

Let $\mathcal{A}$ be a complex Banach algebra with an identity $1$. We define $a\in \mathcal{A}$ has g-Drazin inverse (i.e., generalized Drazin inverse) if there exists $b\in \mathcal{A}$ such that
$$ab=ba, b=bab, a-a^2b\in \mathcal{A}~\mbox{is quasinilpotent}.$$ Such $b$ is unique, if exists, and denote it by $a^d$.  If we replace quasinilpotent in the above definition with nilpotent, then $b$ is called the Drazin inverse of $a$. The g-Drazin inverse plays an important role in matrix and operator theory.
Many authours have been studying this subject from different views (see~\cite{DCW,DW,LCC,M} and~\cite{ZM}). In this paper we provide some new characterizations for the g-Drazin inverse of an element in a Banach algebra. In Section 2, we drop the regular condition for the g-Drazin invertibility off the definition. We then thereby prove that an element $a$ in a Banach algebra $\mathcal{A}$ has g-Drazin inverse if and only if there exist an idempotent $e$, a unit $u$ and a quasinilpotent $w$ which commute each other such that $a=eu+w$. This helps us to generalize~\cite[Theorem 3]{N} and prove that an element $a\in \mathcal{A}$ has g-Drazin inverse if and only if there exists an idempotent $e\in comm(a)$ such that $eae\in [e\mathcal{A}e]^{-1}, (1-e)a(1-e)\in [(1-e)\mathcal{A}(1-e)]^{qnil}$.
It was firstly posed by Campbell that the solutions to singular systems of differential equations is determined by the g-Drazin invertibility of the $2\times 2$ anti-triangular block matrix (see~\cite{B}). The g-Drazin inverse of such special matrices attract many authors (see~\cite{CD,CI,H,LY} and~\cite{ZC}).
In Section 3, we apply the results in Section 2 for certain anti-triangular block matrices over a Banach algebra and provide some necessary and sufficient conditions for such matrices to be g-Drazin invertible. These also extend ~\cite[Theorem 4.1]{CD}, ~\cite[Theorem 2.6]{ZMC} for the g-Drazin inverse.

Throughout the paper, we use $\mathcal{A}^{-1}$ to denote the set of all units in $\mathcal{A}$. $\mathcal{A}^{d}$ indicates the set of all g-Drazin invertible elements in $\mathcal{A}$. Let $a\in \mathcal{A}$. The commutant of $a\in \mathcal{A}$ is defined by $comm(a)=\{x\in
\mathcal{A}~|~xa=ax\}$. ${\Bbb N}$ stands for the set of all natural numbers.

\section{g-Drazin inverse}

The aim of this section is to provide a new characterization of g-Drazin inverse in a Banach algebra. We shall prove that regular condition $"x=xax"$ can be dropped from the definition of g-Drazin inverse. An element $a\in \mathcal{A}$ has strongly Drazin inverse if it is the sum of an idempotent and a quasinilpotent that commute (see~\cite{CM}). We begin with a characterization of strongly Drazin inverse.

\begin{lem} Let $a\in \mathcal{A}$. Then the following are equivalent:\end{lem}
\begin{enumerate}
\item [(1)]{\it $a\in \mathcal{A}$ has strongly g-Drazin inverse.} \vspace{-.5mm}
\item [(2)]{\it $a-a^2\in \mathcal{A}^{qnil}$.}
\end{enumerate}
\begin{proof} See~\cite[Lemma 2.2]{CM}.\end{proof}

We come now to the demonstration for which this paper has been developed.

\begin{thm} Let $a\in \mathcal{A}$. Then the following are equivalent:\end{thm}
\begin{enumerate}
\item [(1)]{\it $a\in \mathcal{A}^d$.} \vspace{-.5mm}
\item [(2)]{\it There exists some $x\in comm(a)$ such that $a-a^2x\in \mathcal{A}^{qnil}$.}
\end{enumerate}
\begin{proof} $(1)\Rightarrow (2)$ This is obvious by choosing $x=a^d$.

$(2)\Rightarrow (1)$ By hypothesis, there exists some $x\in comm(a)$ such that $a-a^2x\in \mathcal{A}^{qnil}$. Set $z=xax$.
Then $z\in comm(a)$. We check that $$\begin{array}{lll}
a-a^2z&=&a-axaxa\\
&=&(1+ax)(a-a^2x)\\
&\in& \mathcal{A}^{qnil},\\
z-z^2a&=&xax-xaxaxax\\
&=&x(a-a^2x)x+xax(a-a^2x)x\\
&\in& \mathcal{A}^{qnil}.
\end{array}$$

$$az-(az)^2=(a-a^2z)z\in \mathcal{A}^{qnil}.$$ Then we have idempotent $e\in comm^2(az)$ such that $az-e\in \mathcal{A}^{qnil}$.
We easily check that
$$(a+1-az)((z+1-az)=1+(a-a^2z)(1-z)+(z-z^2a).$$
Hence, $$\begin{array}{lll}
a+1-e&=&(a+1-az)+(az-e)\in \mathcal{A}^{-1},\\
a(1-e)&=&(a-a^2z)+a(az-e)\in\mathcal{A}^{qnil}.
\end{array}$$
Since $a\in comm(az)$, we have $ea=ae$. That is, $a\in \mathcal{A}$ is quasipolar. Therefore $a\in \mathcal{A}^d$, by~\cite[Theorem 4.2]{K}.\end{proof}

\begin{cor} Let $a\in \mathcal{A}$. Then the following are equivalent:\end{cor}
\begin{enumerate}
\item [(1)]{\it $a\in \mathcal{A}^d$.} \vspace{-.5mm}
\item [(2)]{\it There exists a invertible $u\in comm(a)$ such that $a-a^2u\in \mathcal{A}^{qnil}$.}
\vspace{-.5mm}
\item [(3)]{\it $au$ has strongly g-Drazin inverse for some invertible $u\in comm(a)$.}
\end{enumerate}
\begin{proof} $(1)\Rightarrow (3)$ In view of~\cite[Theorem 4.2]{K}, there exists an idempotent $p\in comm(a)$ such that $u:=a+p\in \mathcal{A}^{-1}$ and
$ap\in \mathcal{A}^{qnil}$. Hence, $ap=a(u-a)\in \mathcal{A}^{qnil}$. Then $a-a^2u^{-1}\in \mathcal{A}^{qnil}$.
Thus $au^{-1}-[au^{-1}]^2\in \mathcal{A}^{qnil}$. Therefore $au$ has strongly g-Drazin inverse by Lemma 2.1.

$(3)\Rightarrow (2)$ In light of Lemma 2.1, $au-(au)^2\in \mathcal{A}^{qnil}$ for some invertible $u\in comm(a)$.
Hence $a-a^2u\in \mathcal{A}^{qnil}$, as required.

$(2)\Rightarrow (1)$ This is obvious by Theorem 2.2.\end{proof}

We are now ready to extend~\cite[Theorem 4.2]{K} as follows.

\begin{cor} Let $a\in \mathcal{A}$. Then the following are equivalent:\end{cor}
\begin{enumerate}
\item [(1)]{\it $a\in \mathcal{A}^d$.} \vspace{-.5mm}
\item [(2)]{\it There exists some $p\in comm(a)$ such that $a+p\in \mathcal{A}^{-1}$ and $ap\in \mathcal{A}^{qnil}$.}
\end{enumerate}
\begin{proof} $(1)\Rightarrow (2)$ This is clear by~\cite[Theorem 4.2]{K}.

$(2)\Rightarrow (1)$ Set $b=(a+p)^{-1}(1-p)$. Then $b\in comm(a)$ and
$$\begin{array}{lll}
ab&=&a(a+p)^{-1}(1-p)\\
&=&(a+p)(a+p)^{-1}(1-p)-p(a+p)^{-1}(1-p)\\
&=&1-p-p(a+p)^{-1}(1-p).
\end{array}$$ In view of~\cite[Lemma 2.11]{ZMC}, we have $$\begin{array}{lll}
a-a^2b&=&a(1-ab)\\
&=&ap[1+(a+p)^{-1}(1-p)]\\
&\in & \mathcal{A}^{qnil}.
\end{array}$$ This completes the proof by Theorem 2.2.\end{proof}

The next result generalizes ~\cite[Proposition 13.1.18]{C}.

\begin{thm} Let $a\in \mathcal{A}$. Then the following are equivalent:\end{thm}
\begin{enumerate}
\item [(1)]{\it $a\in \mathcal{A}^d$.} \vspace{-.5mm}
\item [(2)]{\it There exist an idempotent $e$, a unit $u$ and a quasinilpotent $w$ which commute each other such that $a=eu+w$.}
\end{enumerate}
\begin{proof} $(1)\Rightarrow (2)$ By hypothesis, there exists a invertible $u\in comm(a)$ such that $a-a^2u\in \mathcal{A}^{qnil}$. Then
$(u^{-1}a)^2-u^{-1}a\in \mathcal{A}^{qnil}$. In light of Lemma 2.1, there exists $e^2=e\in comm^2(u^{-1}a)$ such that $w:=u^{-1}a-e\in \mathcal{A}^{qnil}$.
Hence, $a=ue+uw$. Clearly, $eu=ue$ and $ea=ae$; hence, $uw=wu,(ue)(uw)=(uw)(ue)$ and $uw\in \mathcal{A}^{qnil}$, as required.

$(2)\Rightarrow (1)$ Write $a=ue+w$ for an idempotent $e$, an invertible $u$ and a quasinilpotent $w$ which commute each other. Then $(u^{-1}a)^2-u^{-1}a\in \mathcal{A}^{qnil}$. Then $a-a^2u^{-1}\in \mathcal{A}^{qnil}$, as asserted.\end{proof}

\begin{cor} Let $a\in \mathcal{A}^d$. Then $a$ is the sum of two units in $\mathcal{A}$.\end{cor}
\begin{proof} Since $a\in \mathcal{A}^d$, it follows by~\cite[Theorem 3.11]{ZMC} that $\frac{a}{2}\in \mathcal{A}^d$. In view of Theorem 2.5,  there exist an idempotent $e$, a unit $u$ and a quasinilpotent $w$ which commute each other such that $\frac{a}{2}=eu+w$. Hence, $a=2eu+2w=(2e-1)u+u+2w=(2e-1)u+u(1+2u^{-1}w).$ Since $(2e-1)^2=1$ and $1+2u^{-1}w\in \mathcal{A}^{-1}$, $a$ is the sum of two units, as asserted.\end{proof}

\begin{thm} Let $a\in \mathcal{A}$. Then the following are equivalent:\end{thm}
\begin{enumerate}
\item [(1)]{\it $a\in \mathcal{A}^d$.} \vspace{-.5mm}
\item [(2)]{\it There exist an idempotent $e\in comm(a)$ such that $eae\in [e\mathcal{A}e]^{-1}, (1-e)a(1-e)\in [(1-e)\mathcal{A}(1-e)]^{qnil}$.}
\end{enumerate}
\begin{proof}  $(1)\Rightarrow (2)$ By virtue of Theorem 2.5, there exist an idempotent $e$, a unit $u$ and a
quasinilpotent $w$ which commute each other such that $a=eu+w$. Then $eae=eu(1+u^{-1}w)\in [e\mathcal{A}e]^{-1}$. Moreover, we have
$(1-e)a(1-e)=(1-e)w\in [(1-e)\mathcal{A}(1-e)]^{qnil}$, as desired.

$(2)\Rightarrow (1)$ Suppose there exists an idempotent $e\in comm(a)$ such that  $eae\in [e\mathcal{A}e]^{-1}, (1-e)a(1-e)\in [(1-e)\mathcal{A}(1-e)]^{qnil}$. Then
$a=ea+(1-e)a=e[eae+1-e]+(1-e)a$. In view of~\cite[Lemma 2.11]{ZMC}, $(1-e)a\in \mathcal{A}^{qnil}$. Obviously, $eae+1-e\in \mathcal{A}^{-1}$. According to Theorem 2.5,
$a$ has g-Drazin inverse, as asserted.\end{proof}

Let $\alpha\in \mathcal{A}=End(M_k)$. The submodule $P$ of $M$ is $\alpha$-invariant provided that $\alpha (P)\subseteq P$ (see~\cite{N}). We now derive

\begin{cor} Let $\alpha\in \mathcal{A}=End(M_k)$. Then the following are
equivalent:\end{cor}
\begin{enumerate}
\item [(1)]{\it $\alpha\in \mathcal{A}^d$.}
\vspace{-.5mm}
\item [(2)]{\it $M=P\oplus Q$, where $P$ and $Q$ are $\alpha$-invariant, $\alpha|_P\in [End(P)]^{-1}$, $\alpha|_Q\in End(Q)^{qnil}$. The corresponding PQPQ-decomposition looks like
$$\begin{array}{cccccc}
M&=&P&\bigoplus &Q&\\
&\alpha\mid_P=unit &\downarrow &&\downarrow&\alpha\mid_Q=quasinilpotent \\
M&=&P&\bigoplus &Q&.
\end{array}$$}
\end{enumerate}
\begin{proof} $(1)\Rightarrow (2)$ In view of Theorem 2.7, there exist an idempotent $e\in comm(\alpha)$ such that $e\alpha e\in [e\mathcal{A}e]^{-1}, (1-e)\alpha(1-e)\in [(1-e)\mathcal{A}(1-e)]^{qnil}$. Set $P=Me $ and $Q=M(1-e)$. Then
$M=P\oplus Q$. As $e\in comm(\alpha)$, we see that $P$ and $Q$
are $\alpha$-invariant.

Write $(e\alpha e)^{-1}=e\beta e$. Then one easily checks that $[\alpha|_P]^{-1}=\beta|_P$. Let $\gamma\in End(Q)\cap comm(\alpha|_Q)$. We will suffice to prove
$1_Q-\alpha|_Q \gamma\in [End(P)]^{-1}$.
$$\begin{array}{llll}
1_Q-\alpha|_Q \gamma: &Q&\to &Q\\
&p&\mapsto &q-(q)\alpha \gamma .
\end{array}
$$ Define
$\overline{\gamma}:M\to M$ given by $(p+q)\overline{\gamma}
=(q)\gamma$ for any $p\in P,q\in Q$. Set $f=1-e$. If $(q)\big(1_Q-\alpha|_Q \gamma\big)=0$, then $(qf)\big(f
-(f\alpha )f \overline{\gamma}f\big)=0$. As $\alpha f\in
\big(f\mathcal{A}f\big)^{qnil}$, we get $qf=0$. This implies that
$1_Q-\alpha|_Q\in End(Q)$ is an $R$-monomorphism. For any $q\in
Q$. Choose $z=(qf)\big(f -(f\alpha )f
\overline{\gamma}f\big)^{-1}\in Q$. Then $(z)\big(1_Q-\alpha|_Q \gamma\big)=q$; hence, $1_Q-\alpha|_Q
\gamma\in End(Q)$ is an $R$-epimorphism. Thus $1_Q-\alpha|_Q \gamma\in [End(Q)]^{-1}$, and so $\alpha|_Q\in End(Q)^{qnil}$.

$(2)\Rightarrow (1)$ Let $e: M=P\oplus Q\to P$ be the projection on $P$. In view of~\cite[Lemma 2]{N}, $e^2=e\in comm(\alpha)$. Moreover, $P=Me$ and $Q=M(1-e)$.
Since $\alpha|_P\in [End(P)]^{-1}$, we see that $e\alpha e\in (e\mathcal{A}e)^{-1}$.
It follows from $(1-e)\alpha(1-e)\in [(1-e)\mathcal{A}(1-e)]^{qnil}$ that $(1-e)\alpha (1-e)\in [(1-e)\mathcal{A}(1-e)]^{qnil}$. This completes the proof by
Theorem 2.7.\end{proof}

\section{anti-triangular matrices}

In this section we apply Theorem 2.2 to block matrices over a Banach algebra and present necessary and sufficient conditions for the existence of the
g-Drazin inverse for a class of $2\times 2$ anti-triangular block matrices. We now derive

\begin{lem} Let $M=\left(
         \begin{array}{cc}
          1 & 1 \\
          a & 0 \\
          \end{array}
         \right)\in M_2(\mathcal A)$. Then\end{lem}
 \begin{enumerate}
\item [(1)]{\it For any $n\in {\Bbb N}$, $M^n=\left(
         \begin{array}{cc}
          U(n) & U(n-1) \\
          U(n-1)a & U(n-2)b \\
          \end{array}
         \right)$,
where $U(m) =\sum\limits_{i=0}^{[\frac{m}{2}]}
\left(
\begin{array}{c}
m-i\\
i
\end{array}
\right)b^i, m\geq 0; U(-1)=0.$} \vspace{-.5mm}
\item [(2)]{\it $U(n)-U(n-1)=U(n-2)a$ for any $n\in {\Bbb N}$.}
\end{enumerate}
\begin{proof} See~\cite[Proposition 3.7]{CD}.\end{proof}

\begin{lem} Let $a\in \mathcal{A}$. Then the following are equivalent:\end{lem}
\begin{enumerate}
\item [(1)]{\it $a\in \mathcal{A}^d$.} \vspace{-.5mm}
\item [(2)]{\it $\left(
\begin{array}{cc}
1&1\\
a&0
\end{array}
\right)\in M_2(\mathcal{A})^d.$}
\end{enumerate}
\begin{proof} $(1)\Rightarrow (2)$ This is obvious by\cite[Theorem 2.3]{ZM}.

$(2)\Rightarrow (1)$ Write $M^d=\left(
\begin{array}{cc}
x_{11}&x_{12}\\
x_{21}&x_{22}
\end{array}
\right)$. Then $MM^d=M^dM$, and so
$$\left(
\begin{array}{cc}
1&1\\
a&0
\end{array}
\right)\left(
\begin{array}{cc}
x_{11}&x_{12}\\
x_{21}&x_{22}
\end{array}
\right)=\left(
\begin{array}{cc}
x_{11}&x_{12}\\
x_{21}&x_{22}
\end{array}
\right)\left(
\begin{array}{cc}
1&1\\
a&0
\end{array}
\right).$$ Then $$\left(
\begin{array}{cc}
x_{11}+x_{12}a&x_{12}+x_{22}\\
ax_{11}&ax_{12}
\end{array}
\right)=\left(
\begin{array}{cc}
x_{11}+x_{12}a&x_{11}\\
x_{21}+x_{22}a&x_{21}
\end{array}
\right).$$ Hence, we have $$\begin{array}{c}
x_{11}+x_{21}=x_{11}+x_{12}a,\\
ax_{12}=x_{21}.
\end{array}$$ Therefore $ax_{12}=x_{21}=x_{12}a.$

Write $(M^2M^d-M)^n=W_n=\left(
\begin{array}{cc}
\alpha_n&\beta_n\\
\gamma_n&\delta_n
\end{array}
\right) (n\in {\Bbb N})$. Since $M^{n+1}M^d-M^n=W_n$, we see that $${\lim_{n\to \infty}}\parallel W_n\parallel^{\frac{1}{n}}=0,$$ and then
$${\lim_{n\to \infty}}\parallel \left(
\begin{array}{cc}
0&\beta_n\\
0&0
\end{array}
\right)\parallel^{\frac{1}{n}}={\lim_{n\to \infty}}\parallel \left(
\begin{array}{cc}
1&0\\
0&0
\end{array}
\right)W_n\left(
\begin{array}{cc}
0&0\\
0&1
\end{array}
\right)\parallel^{\frac{1}{n}}=0.$$ This implies that $${\lim_{n\to \infty}}\parallel \beta_n\parallel^{\frac{1}{n}}=0.$$
Likewise, $${\lim_{n\to \infty}}\parallel \delta_n\parallel^{\frac{1}{n}}=0.$$
By using Lemma 2.1, we have $$\begin{array}{c}
U(n+1)x_{12}+U(n)x_{22}=U(n-1)+v_0, v_0:=\beta_n,\\
U(n)ax_{12}+U(n-1)ax_{22}=U(n-2)a+v_1, v_1:=\delta_n.
\end{array}$$
Hence, $$U(n+1)ax_{12}+U(n)ax_{22}=U(n-1)a+a\beta_n,$$ and then
$$\begin{array}{c}
U(n-1)a^2x_{12}+U(n-2)a^2x_{22}=U(n-3)a^2+v_2, v_2:==av_0-v_1.
\end{array}$$
Moreover, we have $$\begin{array}{c}
U(n-2)a^3x_{12}+U(n-3)a^3x_{22}=U(n-4)a^3+v_3, v_3:=av_1-v_2.
\end{array}$$
By iteration of this process, we have $$\begin{array}{ll}
U(n-(n-2))a^{n-1}x_{12}+U(n-(n-1))a^{n-1}x_{22}\\
=U(n-n)a^{n-1}+v_{n-1};\\
v_{n-1}:=av_{n-3}-v_{n-2},\\
U(n-(n-1))a^{n}x_{12}+U(n-n)a^{n}x_{22}=U(n-(n+1))a^{n}+v_{n},\\
v_{n}:=av_{n-2}-v_{n-1}.
\end{array}$$  That is,
$$\begin{array}{c}
(1+a)a^{n-1}x_{12}+a^{n-1}x_{22}=a^{n-1}+v_{n-1}, v_{n-1}:=av_{n-3}-v_{n-2};\\
a^{n}x_{12}+a^{n}x_{22}=v_{n}, v_{n}:=av_{n-2}-v_{n-1}.
\end{array}$$
Therefore $$\begin{array}{lll}
a^n&=&a[(1+a)a^{n-1}x_{12}+a^{n-1}x_{22}-v_{n-1}]\\
&=&(1+a)a^{n}x_{12}+a^{n}x_{22}-av_{n-1}\\
&=&(1+a)a^{n}x_{12}+v_n-a^{n}x_{12}-av_{n-1}\\
&=&a^{n+1}x_{12}+v_n-av_{n-1}.
\end{array}$$ Hence,
$$a^n-a^{n+1}x_{12}=v_n-av_{n-1}.$$
We have a recurrence relations $$v_0=\beta_n, v_1=\delta_n, v_{n}=-v_{n-1}+av_{n-2}.$$
By induction, we show that $$\parallel v_n\parallel\leq (1+\parallel a\parallel)^n (\parallel v_0\parallel+\parallel v_1\parallel).$$ Hence we have
$$\parallel v_n-av_{n-1}\parallel\leq (1+\parallel a\parallel)^{n+1}(\parallel \beta_n\parallel+\parallel \delta_n\parallel).$$
Then we get $$\parallel v_n-av_{n-1}\parallel^{\frac{1}{n}}\leq (1+\parallel a\parallel)^{1+\frac{1}{n}}(\parallel \beta_n\parallel^{\frac{1}{n}}
+\parallel \delta_n\parallel^{\frac{1}{n}}).$$ Thus, $${\lim_{n\to \infty}}\parallel a^n-a^{n+1}x_{12}\parallel^{\frac{1}{n}}=0.$$
Since $\parallel (a-a^{2}x_{12})^n\parallel\leq \parallel a^n-a^{n+1}x_{12}\parallel\parallel 1-ax_{12}\parallel^{n-1},$ we deduce that
$${\lim_{n\to \infty}}\parallel (a-a^{2}x_{12})^n\parallel^{\frac{1}{n}}=0.$$ Therefore
$a-a^2x_{12}\in \mathcal{A}^{qnil}$. In light of Theorem 2.2, $a\in \mathcal{A}^d$, as asserted.\end{proof}

\vskip4mm We are ready to extend~\cite[Theorem 2.6]{ZMC} for the g-Drazin inverse.

\begin{thm} Let $M=\left(
\begin{array}{cc}
a&b\\
c&0
\end{array}
\right)\in M_2(\mathcal{A})$. If $a^2=a\in \mathcal{A}$ and $ab=b$, then the following are equivalent:
\end{thm}
\begin{enumerate}
\item [(1)]{\it $M\in M_2(\mathcal{A})^d$.} \vspace{-.5mm}
\item [(2)]{\it $bc\in \mathcal{A}^d.$}\\
\end{enumerate}
\begin{proof} $(1)\Rightarrow (2)$ One easily checks that
$$\begin{array}{c}
\left(
\begin{array}{cc}
a&b\\
c&0
\end{array}
\right)=\left(
\begin{array}{cc}
1&0\\
0&c
\end{array}
\right)\left(
\begin{array}{cc}
a&b\\
1&0
\end{array}
\right),\\
\left(
\begin{array}{cc}
a&bc\\
1&0
\end{array}
\right)=\left(
\begin{array}{cc}
a&b\\
1&0
\end{array}
\right)\left(
\begin{array}{cc}
1&0\\
0&c
\end{array}
\right).
\end{array}$$ By using Cline's formula, $\left(
\begin{array}{cc}
a&bc\\
1&0
\end{array}
\right)$ has g-Drazin inverse.
Moreover, we have
$$\begin{array}{c}
\left(
\begin{array}{cc}
a&bc\\
1&0
\end{array}
\right)=\left(\begin{array}{cc}
a&a\\
1&0
\end{array}
\right)\left(\begin{array}{cc}
1&0\\
0&bc
\end{array}
\right),\\
\left(
\begin{array}{cc}
a&a\\
bc&0
\end{array}
\right)=\left(\begin{array}{cc}
1&0\\
0&bc
\end{array}
\right)\left(
\begin{array}{cc}
a&a\\
1&0
\end{array}
\right).
\end{array}$$ By using Cline's formula again, $\left(
\begin{array}{cc}
a&a\\
bc&0
\end{array}
\right)$ has g-Drazin inverse.
Since $$\left(
\begin{array}{cc}
1&a\\
bc&0
\end{array}
\right)=\left(
\begin{array}{cc}
1-a&0\\
0&0
\end{array}
\right)+\left(
\begin{array}{cc}
a&a\\
bc&0
\end{array}
\right),$$ it follows by~\cite[Theorem 2.2]{DW} that
$\left(
\begin{array}{cc}
1&a\\
bc&0
\end{array}
\right)$ has g-Drazin inverse.
Let $S=\left(
\begin{array}{cc}
1&1\\
bc&0
\end{array}
\right), T=\left(
\begin{array}{cc}
1&0\\
0&a
\end{array}
\right).$ Then
$$ST=\left(
\begin{array}{cc}
1&a\\
bc&0
\end{array}
\right), TS=\left(
\begin{array}{cc}
1&1\\
bc&0
\end{array}
\right).$$ In view of Cline's formula, $\left(
\begin{array}{cc}
1&1\\
bc&0
\end{array}
\right)$ has g-Drazin inverse.
In light of Lemma 3.2, $bc\in \mathcal{A}^d$, as asserted.

$(2)\Rightarrow (1)$ Since $bc=abc\in \mathcal{A}^d$, it follows by Cline's formula that $bca$ has g-Drazin inverse. In light of Lemma 3.2,
$\left(
\begin{array}{cc}
1&1\\
bca&0
\end{array}
\right)$ has g-Drazin inverse. As $$\left(
\begin{array}{cc}
1&1\\
bca&0
\end{array}
\right)\left(
\begin{array}{cc}
a&0\\
0&a
\end{array}
\right)=\left(
\begin{array}{cc}
a&0\\
0&a
\end{array}
\right)\left(
\begin{array}{cc}
1&1\\
bca&0
\end{array}
\right),$$ it follows by~\cite[Theorem 5.5]{K} that $\left(
\begin{array}{cc}
a&a\\
bca&0
\end{array}
\right)$ has g-Drazin inverse.
Since $$\left(
\begin{array}{cc}
a&a\\
bc&0
\end{array}
\right)=\left(
\begin{array}{cc}
0&0\\
bc(1-a)&0
\end{array}
\right)+\left(
\begin{array}{cc}
a&a\\
bca&0
\end{array}
\right),$$ it follows by~\cite[Theorem 2.2]{DW} that $\left(
\begin{array}{cc}
a&a\\
bc&0
\end{array}
\right)$ has g-Drazin inverse. We easily check that
$$\begin{array}{c}
\left(
\begin{array}{cc}
a&a\\
bc&0
\end{array}
\right)=\left(
\begin{array}{cc}
1&0\\
0&bc
\end{array}
\right)\left(
\begin{array}{cc}
a&a\\
1&0
\end{array}
\right),\\
\left(
\begin{array}{cc}
a&bc\\
1&0
\end{array}
\right)=\left(
\begin{array}{cc}
a&a\\
1&0
\end{array}
\right)\left(
\begin{array}{cc}
1&0\\
0&bc
\end{array}
\right).
\end{array}$$ In view of Cline's formula,
$\left(
\begin{array}{cc}
a&bc\\
1&0
\end{array}
\right)$ has g-Drazin inverse.
Furthermore, we have
$$\begin{array}{c}
\left(
\begin{array}{cc}
a&b\\
c&0
\end{array}
\right)=\left(
\begin{array}{cc}
1&0\\
0&c
\end{array}
\right)\left(
\begin{array}{cc}
a&b\\
1&0
\end{array}
\right),\\
\left(
\begin{array}{cc}
a&bc\\
1&0
\end{array}
\right)=\left(
\begin{array}{cc}
a&b\\
1&0
\end{array}
\right)\left(
\begin{array}{cc}
1&0\\
0&c
\end{array}
\right).
\end{array}$$
By using Cline's formula again, we conclude that $M$ gas g-Drazin inverse.\end{proof}

\begin{cor} Let $M=\left(
\begin{array}{cc}
a&a\\
b&0
\end{array}
\right)\in M_2(\mathcal{A})$. If $a^2=a\in \mathcal{A}$, then the following are equivalent:
\end{cor}
\begin{enumerate}
\item [(1)]{\it $M\in M_2(\mathcal{A})^d$.} \vspace{-.5mm}
\item [(2)]{\it $ab\in \mathcal{A}^d.$}\\
\end{enumerate}
\begin{proof} This is obvious by Theorem 3.3.\end{proof}

\begin{lem} Let $M=\left(
\begin{array}{cc}
a&b\\
c&0
\end{array}
\right)\in M_2(\mathcal{A})$. If $a\in \mathcal{A}^d, caa^d=c$ and $a^dbc=bca^d$, then the following are equivalent:
\end{lem}
\begin{enumerate}
\item [(1)]{\it $M\in M_2(\mathcal{A})^d$.} \vspace{-.5mm}
\item [(2)]{\it $bc\in \mathcal{A}^d.$}\\
\end{enumerate}
\begin{proof} $2)\Rightarrow (1)$ Since $a^dbc=bca^d$, it follows by~\cite[Theorem 5.5]{K} that $(a^d)^2bc\in \mathcal{A}^d$.
In view of Lemma 3.2,
$$\left(
\begin{array}{cc}
1&1\\
(a^d)^2bc&0
\end{array}
\right)\in M_2(\mathcal{A})^d.$$
We easily check that
$$\left(
\begin{array}{cc}
a&0\\
0&a
\end{array}
\right)\left(
\begin{array}{cc}
1&1\\
(a^d)^2bc&0
\end{array}
\right)=\left(
\begin{array}{cc}
1&1\\
(a^d)^2bc&0
\end{array}
\right)\left(
\begin{array}{cc}
a&0\\
0&a
\end{array}
\right),$$ we see that $$\left(
\begin{array}{cc}
a&0\\
0&a
\end{array}
\right)\left(
\begin{array}{cc}
1&1\\
(a^d)^2bc&0
\end{array}
\right)\in M_2(\mathcal{A})^d.$$
This shows that
$$\left(
\begin{array}{cc}
a&0\\
0&b
\end{array}
\right)\left(
\begin{array}{cc}
1&1\\
ca^d&0
\end{array}
\right)\in M_2(\mathcal{A})^d.$$
By using Cline's formula,
$$M=\left(
\begin{array}{cc}
1&1\\
ca^d&0
\end{array}
\right)\left(
\begin{array}{cc}
a&0\\
0&b
\end{array}
\right)\in M_2(\mathcal{A})^d.$$

$(1)\Rightarrow (2)$ Since $M$ has g-Drazin inverse, we prove that
$$\left(
\begin{array}{cc}
a&1\\
bc&0
\end{array}
\right)\in M_2(\mathcal{A})^d.$$ Since $a^d(bc)=(bc)a^d$, by virtue of ~\cite[Theorem 3.1]{ZMC}, we have
$$\left(
\begin{array}{cc}
a^da&a^d\\
a^dbc&0
\end{array}
\right)=\left(
\begin{array}{cc}
a^d&0\\
0&a^d
\end{array}
\right)\left(
\begin{array}{cc}
a&1\\
bc&0
\end{array}
\right)\in M_2(\mathcal{A})^d.$$ By using Cline's formula,
$$\left(
\begin{array}{cc}
a^da&aa^d\\
(a^d)^2bc&0
\end{array}
\right)=\left(
\begin{array}{cc}
1&0\\
0&a^d
\end{array}
\right)\left(
\begin{array}{cc}
a^da&a^d\\
a^dbc&0
\end{array}
\right)\left(
\begin{array}{cc}
1&0\\
0&a
\end{array}
\right)\in M_2(\mathcal{A})^d.$$
One easily checks that
$$\left(
\begin{array}{cc}
1&1\\
(a^d)^2bc&0
\end{array}
\right)=\left(
\begin{array}{cc}
a^{\pi}&a^{\pi}\\
0&0
\end{array}
\right)+\left(
\begin{array}{cc}
aa^d&aa^d\\
(a^d)^2bc&0
\end{array}
\right).$$ Hence, $$\left(
\begin{array}{cc}
1&1\\
(a^d)^2bc&0
\end{array}
\right)\in M_2(\mathcal{A})^d.$$ In light of Lemma 2.2,
$(a^d)^2bc\in \mathcal{A}^d$. Since $a(a^d)^2bc=(a^d)^2bca$, we see that $a^2(a^d)^2bc=(a^d)^2bca^2$. In view of~\cite[Theorem 3.1]{ZMC},
$$bc=bc(a^d)^2a^2=(a^d)^2bca^2\in \mathcal{A}^d,$$ as asserted.\end{proof}

The following result is a generalization of~\cite[Theorem 4.1]{CD} for the g-Drazin inverse.

\begin{thm} Let $M=\left(
\begin{array}{cc}
a&b\\
c&0
\end{array}
\right)\in M_2(\mathcal{A})$. If $a\in \mathcal{A}^d, bca^{\pi}=0$ and $a^dbc=bca^d$, then the following are equivalent:
\end{thm}
\begin{enumerate}
\item [(1)]{\it $M\in M_2(\mathcal{A})^d$.} \vspace{-.5mm}
\item [(2)]{\it $bc\in \mathcal{A}^d.$}\\
\end{enumerate}
\begin{proof} $(2)\Rightarrow (1)$ Let $c'=caa^d$. Since $bca^{\pi}=0$, we have $bc=bcaa^d$. We see that
$$M=P+Q, P=\left(
\begin{array}{cc}
a&b\\
c'&0
\end{array}
\right), Q=\left(
\begin{array}{cc}
0&0\\
ca^{\pi}&0
\end{array}
\right).$$ Clearly, $PQ=0$ and $Q^2=0$. Since $c'a^{\pi}=0, a^dbc'=bc'a^d$ and $bc'=bc\in \mathcal{A}^d$, it follows by Lemma 3.5 that
$P$ has g-Drazin inverse. In light of ~\cite[Theorem 2.2]{DW}, $M$ has g-Drazin inverse, as required.

$(1)\Rightarrow (2)$ One easily checks that
$$\left(
\begin{array}{cc}
a&b\\
c'&0
\end{array}
\right)=M+N, N=\left(
\begin{array}{cc}
0&0\\
-ca^{\pi}&0
\end{array}
\right).$$ Clearly, $MN=0$ and $N^2=0$. In view of~\cite[Theorem 2.2]{DW}, $\left(
\begin{array}{cc}
a&b\\
c'&0
\end{array}
\right)$ has g-Drazin inverse. Moreover, $c'a^{\pi}=0, a^dbc'=bc'a^d$ and $bc'=bc\in \mathcal{A}^d$. According to Lemma 3.5,
$bc=bc'$ has g-Drazin inverse, as asserted.\end{proof}

\begin{cor} Let $M=\left(
\begin{array}{cc}
a&b\\
c&0
\end{array}
\right)\in M_2(\mathcal{A})$. If $a\in \mathcal{A}^d, a^{\pi}bc=0$ and $abc=bca$, then the following are equivalent:
\end{cor}
\begin{enumerate}
\item [(1)]{\it $M\in M_2(\mathcal{A})^d$.} \vspace{-.5mm}
\item [(2)]{\it $bc\in \mathcal{A}^d.$}\\
\end{enumerate}
\begin{proof} Since $a(bc)=(bc)a$ and $a$ has g-Drazin inverse, by ~|cite[Theorem 4.4]{K}, $a^d(bc)=(bc)a^d$, and so $0=a^{\pi}bc=(1-aa^d)bc=bc(1-aa^d)=bca^{\pi}$. The corollary is therefore established by Theorem 3.6.\end{proof}

\end{document}